%% file: main.tex
\documentclass[12pt]{amsart}

\input{eng}

\input{eng_theorems}
\input{preamble}

\input{listmain}

\theoremstyle{plain}
    \newtheorem{stern}[theorem]{Sternfeld's Arrays Theorem}
\theoremstyle{definition}
    \newtheorem{arno}[theorem]{Arnold's Problem}

\newcommand\tonset[1]{\ensuremath{\{1,\ldots,#1\}}}
\newcommand\fromto[2]{\ensuremath{#1_1,\ldots,#1_{#2}}}

\newcommand{\jonly}[1]{}
\newcommand{\aronly}[1]{#1}

\begin{document}

\title{A simpler proof of Sternfeld's Theorem}
%\author{S. Djenjer and A. Skopenkov}
\author{S. Dzhenzher}
\date{}

\thanks{Moscow Institute of Physics and Technology. 
\newline
sdjenjer@yandex.ru
\newline
I would like to thank Arkadiy Skopenkov and Vitaliy Kurlin for many useful suggestions.}

\maketitle

\begin{abstract}
    In Sternfeld's work on Kolmogorov's Superposition Theorem appeared the com\-bi\-natorial-geo\-metric notion of a basic set and a certain kind of arrays.
    A~subset $X \subset \R^n$ is \emph{basic} if any continuous function $X\to\R$ could be represented as the sum of compositions of continuous functions $\R\to\R$ and projections to the coordinate axes.
    The~definition of a~Sternfeld array is presented in this paper.
    %for any continuous function $f\colon K\to\R$ there are continuous functions
    %$\varphi_1, \ldots, \varphi_d\colon\R\to\R$ such that $f(x) = \sum\limits_{k=1}^d \varphi_k(x_k)$ for each $x \in K$.
    
    \textbf{Sternfeld's Arrays Theorem.}
    \emph{If a closed bounded subset $X \subset \R^{2n}$ contains Sternfeld arrays of arbitrary large size then $X$ is not basic}.

    The paper provides a simpler proof of this theorem.
\end{abstract}

\section{Introduction}

Hilbert's 13th problem is as follows: \emph{can the equation $x^7 + ax^3 + bx^2 + cx + 1 = 0$ of degree seven be solved without using functions of three variables?} The problem is solved by Kolmogorov's Superposition Theorem (Remark~\ref{rem:motivation}.\ref{en:motivation:th-kolm-super}; for details see \cite{Ar, Ko, Sk}).
% In this theorem the number $2n+1$ cannot be decreased.
% This follows by Sternfeld's Dimension Theorem (Remark~\ref{rem:motivation}.\ref{en:motivation:th-stern-dim}; for details see Remark~\ref{rem:motivation}.\ref{en:motivation:nondecr}).
% The proof of the latter theorem in the case $n=2$ uses Sternfeld's Arrays Theorem~\ref{th:basicSubset} (for details see~Remark~\ref{rem:motivation}.\ref{en:motivation:st-dim-proof}).
In the future works related to this theorem Sternfeld's Arrays Theorem~\ref{th:basicSubset} appeared to be useful (for details and motivation see
Remark~\ref{rem:motivation}.\ref{en:motivation:th-kolm-super},\ref{en:motivation:nondecr},\ref{en:motivation:th-stern-dim},\ref{en:motivation:st-dim-proof}).
The main result of this text is~a~simpler proof of Sternfeld's Arrays Theorem~\ref{th:basicSubset}.
This theorem is a necessary condition for a set being basic (and is interesting for applications in the case $n=2$).

A subset $X \subset \R^n$ is \emph{\textbf{basic}} if for any continuous function $f\colon X\to\R$ there are continuous functions $\varphi_1, \ldots, \varphi_n\colon\R\to\R$ such that $f(x_1, \ldots, x_n) = \varphi_1(x_1) + \ldots + \varphi_n(x_n)$ for any $(x_1, \ldots, x_n) \in X$.

Remarks are not used in formulations and in proof of the main result (and can be omitted).

\begin{remark}[Motivation]
\label{rem:motivation}
\hspace{0pt}
\begin{remarkenumi}
    \item \label{en:motivation:th-kolm-super}
    \textbf{Kolmogorov's Theorem.}
    \emph{For any integer $n>1$ there are numbers $\alpha_1,\ldots,\alpha_n$ 
        and continuous functions $g_1,\ldots,g_{2n+1}\colon[\,0,1\,]\to[\,0,1\,]$ such that for any continuous function $f\colon[\,0,1\,]^n\to\R$ 
        there is a continuous function $h\colon\R\to\R$ 
        such that for any $x_1, \ldots, x_n \in [\,0,1\,]$
        \[
            f(x_1, \ldots, x_n) = \sum_{k=1}^{2n+1} h\left(\sum_{i=1}^n \alpha_i g_k(x_i)\right).
        \]
    }
    
    \item \label{en:motivation:kolm-impl-bas}
    \emph{Corollary.}
    There is an embedding $[\,0,1\,]^n \to \R^{2n+1}$ whose image is basic.
    % Let us prove this implication.
    %Using the numbers $\alpha_1, \ldots, \alpha_n$ and the functions $g_1, \ldots, g_{2n+1}$ from Kolmogorov's Theorem, define the continuous function
    
    \emph{Proof.}
    The required embedding $\psi = (\psi_1, \ldots, \psi_{2n+1})\colon [\,0,1\,]^{n}\to\R^{2n+1}$ is defined by
    \[
        \psi_k(x_1, \ldots, x_n) := \sum_{i=1}^n \alpha_i g_k(x_i)\quad\text{for each $x_1, \ldots, x_n \in [\,0,1\,]$},\qquad k=1,\ldots,2n+1.
    \]
    Let us show that $\psi\bigl([\,0,1\,]^n\bigr) \subset \R^{2n+1}$ is basic.
    Suppose we have a continuous function $\widehat f\colon \psi\bigl([\,0,1\,]^n\bigr)\to\R$.
    Define the function $f\colon [\,0,1\,]^n\to\R$ by $f(x) := \widehat f\bigl(\psi (x)\bigr)$.
    Take~the~function $h\colon\R\to\R$ from Kolmogorov's Theorem, applied to~$f$.
    Take $\varphi_1 = \ldots = \varphi_{2n+1} = h$ in~the~definition of~a~basic set.
    For each $y = (y_1, \ldots, y_{2n+1}) \in \psi\bigl([\,0,1\,]^n\bigr)$ the preimage $\psi^{-1}(y)$ is~well defined, so
    \[
        \hat f(y) = f\!\left(\psi^{-1}(y)\right) =
        \sum_{k=1}^{2n+1} h \bigl( \psi_k \bigl( \psi^{-1}(y)\bigr)\bigr) =
        \sum_{k=1}^{2n+1} \varphi_k (y_k).
    \]
    Then the set $\psi\bigl([\,0,1\,]^n\bigr)$ is basic.

    \item \label{en:motivation:nondecr}
    The number $2n+1$ in Kolmogorov's Theorem \eqref{en:motivation:th-kolm-super} cannot be decreased.

    \emph{Proof.}
    On the contrary, suppose that this number can be decreased down to $2n$.
    Then analogously to corollary~\eqref{en:motivation:kolm-impl-bas} there is an embedding $[\,0,1\,]^n \to \R^{2n}$ whose image is basic.
    This contradicts Sternfeld's Dimension Theorem~\eqref{en:motivation:th-stern-dim}.
    
    \item \label{en:motivation:th-stern-dim}
    \textbf{Sternfeld's Dimension Theorem.} \cite[Theorem~5]{St85}
    \emph{No subset of $\R^{2n}$ homeomorphic to $[\,0,1\,]^n$ is basic for $n > 1$.}
    
    On generalizations of \eqref{en:motivation:th-kolm-super}, \eqref{en:motivation:th-stern-dim} see
    \aronly{Remark~\ref{rem:generalSternfeld}}\jonly{\cite[Remark~3.2]{Dz22}}.
    
    \item
    The analogue of~\eqref{en:motivation:th-stern-dim} for $n=1$ is false.
    E.g. the segment $0 \times [\,0,1\,]$ is basic, cf.~\cite[Example~2.26 and \S6]{St89}.
    For the~characterization of basic subsets of the plane see Arnold's Problem~\ref{prob:Arnold},
    \aronly{Remark~\ref{rem:ArnoldDiscussion}}\jonly{\cite[Remark~3.3]{Dz22}}
    and \cite{Mi}.

    \item \label{en:motivation:st-dim-proof}
    Sternfeld's Dimension Theorem~\eqref{en:motivation:th-stern-dim} in the case $n=2$ is implied by Sternfeld's Arrays Theorem~\ref{th:basicSubset} and the following fact: \emph{any subset of $\R^4$ homeomorphic to $[\,0,1\,]^2$ contains Sternfeld arrays of arbitrary large size.}
    See proof of this fact in e.g. \cite[\S5, proof of Theorem~4 for $n=2$]{St89}.
    The analogous fact \cite[Theorem~7]{St85} in $\R^{2n}$ with $n>2$ uses the more general notion \cite[Definition~2.4]{St85} instead of Sternfeld arrays
    (for details see \aronly{Remark~\ref{rem:otherProofs}.\ref{cnt:compareDefs}}\jonly{\cite[Remark~3.1.d]{Dz22}}).
\end{remarkenumi}    
\end{remark}

\begin{figure}[ht]
    \[
    \xymatrixrowsep{3mm}
    \xymatrixcolsep{3mm}
    \xymatrix{
        &&&&&&&& \\
        &   a_2\ar@{.}[rrr]\ar@{.}[dd] &&& a_3\ar@{.}[dddd] &&&& \\
        &&&&& a_7\ar@{.}[dd]\ar@{.}[rr] && a_6\ar@{.}[ddd] & \\
        & a_1 &&&&&&& \\
        && a_9\ar@{.}[dd]\ar@{.}[rrr] &&& a_8 &&& \\
        &&&& a_4\ar@{.}[rrr] &&& a_5 & \\
        && a_{10} &&&&&& \\
        \ar@{->}[uuuuuuu]^(.95){x_2}
        \ar@{->}[rrrrrrrr]_(.95){x_1} &&&&&&&&
    }
    \]
    \caption{A Sternfeld array of size $10$ in the plane $\R^2$}
    \centering
    \label{fig:arrayR2}
\end{figure}
\begin{figure}[ht]
    \[
    \xymatrixrowsep{5mm}
    \xymatrixcolsep{5mm}
    \xymatrix{
        a_1 \ar@{-}[r]^1 & a_2 \ar@{-}[r]^2 & a_3 \ar@{-}[r]^1 & a_4 \ar@{-}[r]^2 & a_5 \ar@{-}[r]^1 & a_6 \ar@{-}[r]^2 & a_7 \ar@{-}[r]^1 & a_8 \ar@{-}[r]^2 & a_9 \ar@{-}[r]^1 & a_{10}
    }
    \]
    \caption{
            Another view on Sternfeld arrays of size $10$ in the plane $\R^2$.
            The sign $\stackrel{k}{\text{--}}$ between points means that $k$th coordinates of these points are equal}
    \label{fig:arrayGridR2}
    \centering
\end{figure}

\begin{remark}[The definition of a Sternfeld array for the particular case of the plane]
\label{rem:definitionPlane}
    A~\emph{Sternfeld array of size $m$ in the plane~$\R^2$} is a sequence $(\fromto{a}{m})$ of pairwise distinct points in $\R^2$ such that the segment~$[\,a_1,a_2\,]$ is parallel to the ordinate axis and for each $i \in \{2, \ldots, m-1\}$
    the segments~$[\,a_{i-1}, a_i\,]$ and~$[\,a_i, a_{i+1}\,]$ form the right angle.
%The~definition is a specific case of \cite[\S 2, Definition 2]{MKT}.

See an~example of a~Sternfeld array in the plane in Figure~\ref{fig:arrayR2}.

\end{remark}

\begin{figure}[ht]
    \includegraphics[height=70mm]{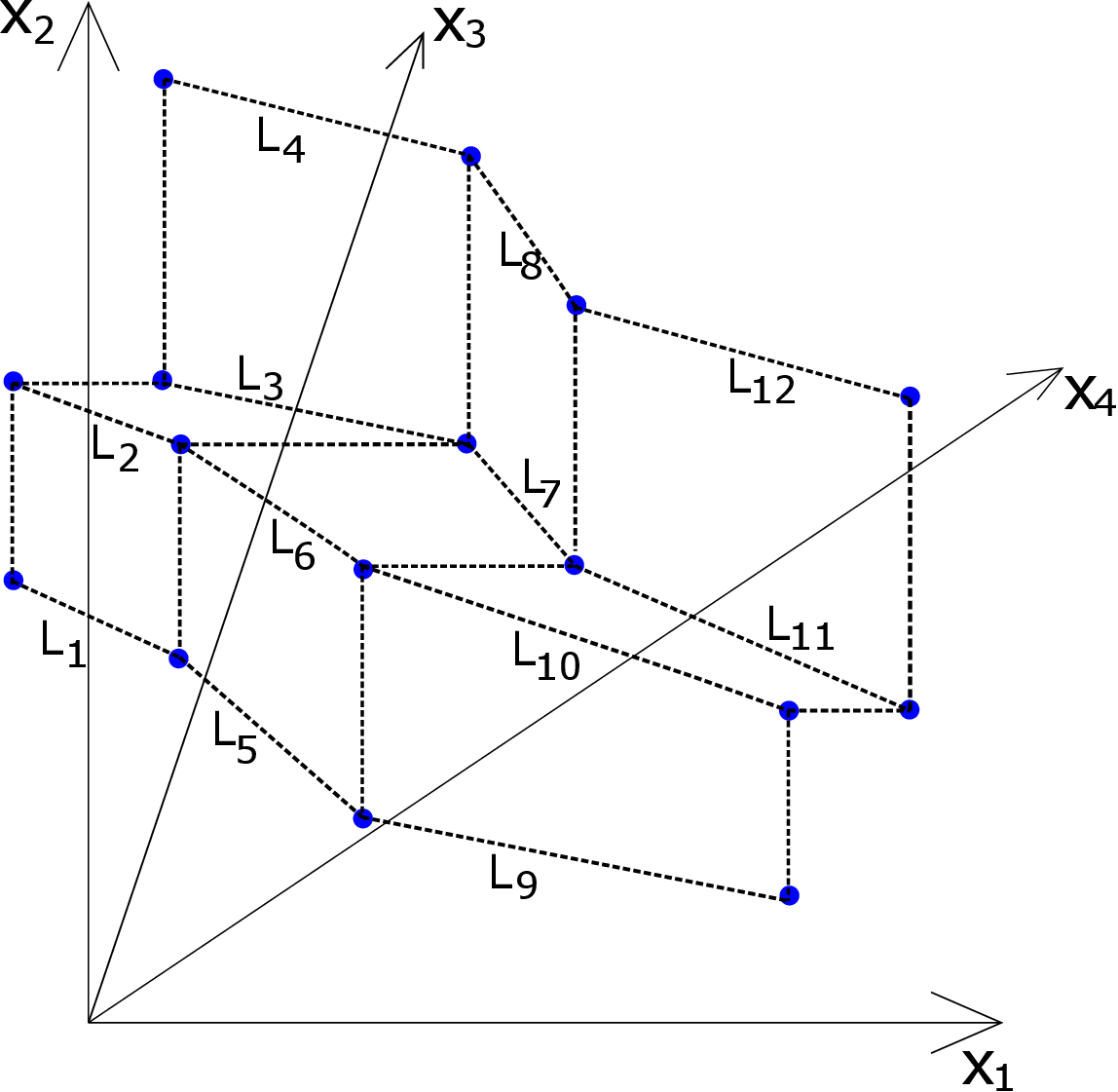}
    \caption{A Sternfeld array of size $4$ in $\R^4$.
            Each vertical line lies in the hyperplane $x_1 = const$.
            Each horizontal line lies in the hyperplane $x_2 = const$.
            Each of the lines $L_1, L_2, L_3, L_4, L_9, L_{10}, L_{11}, L_{12}$ lies in the hyperplane $x_3 = const$.
            Each of the lines $L_5, L_6, L_7, L_8$ lies in the hyperplane $x_4 = const$
            }
    \label{fig:arrayR4}
\end{figure}
\begin{figure}[ht]
    \[
        \xymatrix{
            a_{1,1} \ar@{-}[r]^1\ar@{-}[d]_3 & a_{1,2} \ar@{-}[r]^2\ar@{-}[d]_3 & a_{1,3} \ar@{-}[r]^1\ar@{-}[d]_3 & a_{1,4}\ar@{-}[d]_3 \\
            a_{2,1} \ar@{-}[r]^1\ar@{-}[d]_4 & a_{2,2} \ar@{-}[r]^2\ar@{-}[d]_4 & a_{2,3} \ar@{-}[r]^1\ar@{-}[d]_4 & a_{2,4}\ar@{-}[d]_4 \\
            a_{3,1} \ar@{-}[r]^1\ar@{-}[d]_3 & a_{3,2} \ar@{-}[r]^2\ar@{-}[d]_3 & a_{3,3} \ar@{-}[r]^1\ar@{-}[d]_3 & a_{3,4}\ar@{-}[d]_3 \\
            a_{4,1} \ar@{-}[r]^1             & a_{4,2} \ar@{-}[r]^2             & a_{4,3} \ar@{-}[r]^1             & a_{4,4}
        }
    \]
    \caption{
    %       Grid of a
            Another view on Sternfeld arrays of size $4$ in $\R^4$.
            The sign $\stackrel{k}{\text{--}}$ between points means that $k$th coordinates of these points are equal.
            Here we identify $a_{s,r}$ with $a_{(s,r)}$ for $s,r\in[4]$
            }
    \label{fig:arrayGridR4}
\end{figure}

Recall that $[n] = \tonset{n}$.
For $k\in [n]$ let $y_k$ be the $k$th coordinate of a point $y \in Y^n$ for some $Y$.

Let $m$ be an integer.
Recall that voxels $i,j\in[m]^n$ of an $n$-dimensional array are \emph{adjacent by the coordinate $t \in [n]$} if
%for some $t \in [n]$ we have
$j = (i_1, \ldots, i_{t-1}, i_t \pm 1, i_{t+1}, \ldots, i_n)$.
For a pair $\{i,j\}$ of voxels adjacent by the coordinate $t \in [n]$ denote $\xi(i,j,t)=\xi(t):=2t - \rho_2(i_t)$, where $j_t = i_t + 1$ and $\rho_2\colon\Z\to\{0,1\}$ is defined by $\rho_2(x) = x \bmod 2$.

An $n$-dimensional array $(a_i)_{i\in [m]^n}$, consisting of pairwise distinct points in $\R^{2n}$, is called a {\itshape\bfseries Sternfeld array of size $m$ in $\R^{2n}$} if for each pair $\{i,j\}$ of
%adjacent voxels and for the unique $t \in [n]$ such that $j_t = i_t + 1$ we have
voxels adjacent by a~coordinate~$t \in [n]$ we have
%$(a_{i})_{2t - \rho_2(i_t)} = (a_{j})_{2t - \rho_2(i_t)}$, where $\rho_2\colon\Z\to\{0,1\}$ is defined by $\rho_2(x) = x \bmod 2$.
$a_{i,\xi(t)} = a_{j,\xi(t)}$.

\begin{remark}[Examples]
\label{rem:arrayExamplesRn}
    The concept of a Sternfeld array is described in \cite[\S5, item~12, the~figure]{St89} (compare to Figure~\ref{fig:arrayGridR4}).

    See an~example of a Sternfeld array for $n=2$ in Figure~\ref{fig:arrayR4}.
    
    A simple example of a Sternfeld array of size $2$ in $\R^{2n}$ is the set of vertices of the $2n$-dimensional unit cube, i.e. $\{0,1\}^{2n} \subset \R^{2n}$.

    Another examples of Sternfeld arrays are rook cycles in \cite[Assertion~3.2]{ADN+}.

    \aronly{An example of an array embedded into some plane $\R^2 \subset \R^4$ is as follows.
    Consider the plane given by the equation $x_1 + x_2 = x_3 + x_4 = 0$.
    Then it contains the array defined by
    $a_{2n+m,2s+r} = ((-1)^m n, (-1)^{m+1}n, (-1)^r s, (-1)^{r+1}s)$ for $n,s=0,1,2,\ldots$ and $m,r \in \{1,2\}$.
    }

    \aronly{For $k,l \in [n]$ let $x^{k,l} = (x_k, x_l)$ be the projection of a point $x \in \R^n$ to the plane defined by $k$th and $l$th coordinates.
    Another example of a Sternfeld array of size $m$ in~$\R^{2n}$ is
    an $n$-dimensional array $(a_i)_{i\in [m]^n}$, consisting of pairwise distinct points in~$\R^{2n}$,
    such that for every $t \in [n]$ and $i_1, \ldots, i_{t-1}, i_{t+1}, \ldots, i_n \in [m]$ the sequence $\Bigl(a_{(\fromto{i}{n})}^{2t-1,2t}\Bigr)_{i_t \in [m]}$ is a Sternfeld array of size $m$ in the plane (Remark~\ref{rem:definitionPlane}).
    However, not all Sternfeld arrays fit this example, because for a Sternfeld array $(a_i)_{i\in [m]^n}$ points $a_i^{2t-1,2t}$ and $a_j^{2t-1,2t}$ may coincide for some voxels $i,j$ and some $t$, so the sequence $\Bigl(a_{(\fromto{i}{n})}^{2t-1,2t}\Bigr)_{i_t \in [m]}$ would not be a Sternfeld array of size $m$ in the plane.}
\end{remark}

\begin{stern}
% [Sternfeld's Arrays Theorem]
\label{th:basicSubset}
    If a closed bounded subset $X \subset \R^{2n}$ contains\footnote{An array is a map $[m]^n \to \R^{2n}$, so here `contains' means `contains the images of'.}
    Sternfeld arrays of arbitrary large size then $X$ is not basic.
\end{stern}

Theorem~\ref{th:basicSubset} for $n = 1$ follows from \cite[\S2, Lemma~23, \emph{(iii)}]{St89}.
Theorem~\ref{th:basicSubset} for $n = 2$ is essentially proved in \cite[\S5]{St89}.
Theorem~\ref{th:basicSubset} for $n \geqslant 2$ follows from \cite[\S2, Proof of Theorem~5]{St85} and
\aronly{Remark~\ref{rem:otherProofs}.\ref{cnt:compareDefs}}\jonly{\cite[Remark~3.1.d]{Dz22}}.
See detailed comparison of our proof to Sternfeld's in
\aronly{Remark~\ref{rem:otherProofs}.\ref{cnt:comparisonWithSternfeld}}\jonly{\cite[Remark~3.1.a]{Dz22}},
and of our proof to \cite{MKT} in
\aronly{Remark~\ref{rem:otherProofs}.\ref{cnt:comparisonWithMKT}}\jonly{\cite[Remark~3.1.b]{Dz22}}.

\begin{arno}
\label{prob:Arnold}
    In \cite{Ar} Arnold formulated the following problem: \emph{which subsets of the plane are basic?}
    This problem was solved by Sternfeld
    (for details see \aronly{Remark~\ref{rem:ArnoldDiscussion}}\jonly{\cite[Remark~3.3]{Dz22}}).
    However, the following problem is open: \emph{which subsets of $\R^n$ are basic?}

    The discrete version of the problem is a particular case of finite subsets.
    For details see~\cite{Sk, Re, NR}.
\end{arno}

For \emph{basic embeddability} in $\R^n$ see \cite{St85, Sk, CRS, Ku}.

\section{Proof of Sternfeld's Arrays Theorem \ref{th:basicSubset}}

\begin{lemma}\label{lem:LongSternfeld array}
    Suppose $(c_i)_{i \in [m]^n}$ is an array in $\R^{2n}$ of even size $m$ such that
    
    \begin{subequations}
        \renewcommand{\theequation}{\ref{lem:LongSternfeld array}.\arabic{equation}}
        \begin{align}
            &
            \left.
            \begin{aligned}
                &\text{for each pair $\{i,j\}$ of voxels adjacent by the coordinate $t \in [n]$}
                \\
                & \text{we have}\quad
                %c\bigl(i, 2t - \rho_2(i_t)\bigr) + c\bigl(j, 2t - \rho_2(i_t)\bigr) = 0;
                c_{i,\xi(t)} + c_{j,\xi(t)} = 0;
            \end{aligned}
            \right.
            \label{eq:adjacentCells}
            \\
            &
            \text{for each voxel}\quad i\in [m]^n\quad\text{we have}\quad
            \sum_{k=1}^{2n}c_{i,k} > \frac{1}{2}.
            \label{eq:halfClose}
        \end{align}
    \end{subequations}

     Then $c_{i,k} > \dfrac{m}{4n}$ for some $i,k$.
\end{lemma}

\setcounter{mcnt}{0}
\renewcommand{\themcnt}{\text{(\ref{lem:LongSternfeld array}.\alph{mcnt})}}
\begin{proof}
    The lemma follows because
    \[
        \frac{m^n}{2}
        \stackrel{\juststep}{<}
        \sum_{(i,k) \in [m]^n \times [2n]} c_{i,k}
        % \createLabelTwo{\label{lemeq:b:LemDecomposition}}
        \stackrel{\juststep}{=}
        \left(\sum_{(i,k) \in A} + \sum_{(i,k) \in B} + \sum_{(i,k) \in C} \right) c_{i,k}
        \stackrel{\juststep}{=}
        \sum_{(i,k) \in C} c_{i,k}.
    \]
\setcounter{mcnt}{0}
    Here
    \begin{itemize}
        \item \cnt follows by \eqref{eq:halfClose};
        
        \item
        \begin{align*}
            A =&
            \bigl\{(\lefteqn{i}\phantom{i + e_t}, 2t-1) : i \in [m]^n,\, t \in [n],\, i_t\; \text{is odd} \bigr\}
            \sqcup \\
            \sqcup&
            \bigl\{(i + e_t, 2t-1) : i \in [m]^n,\, t \in [n],\, i_t\; \text{is odd} \bigr\}, \\
            B =&
            \bigl\{(\lefteqn{i}\phantom{i + e_t}, 2t) : i \in [m]^n,\, t \in [n],\, i_t\; \text{is even},\;i_t < m \bigr\}
            \sqcup \\
            \sqcup&
            \bigl\{(i + e_t, 2t) : i \in [m]^n,\, t \in [n],\, i_t\; \text{is even},\;i_t < m \bigr\}, \\
            C =&
            \bigl\{(i, 2t) : i \in [m]^n,\, t \in [n],\, i_t \in \{1,m\} \bigr\},
        \end{align*}
        where $e_t = (0, \ldots, 0, 1, 0, \ldots, 0) \in [m]^n$ is a voxel (that is the $t$-th vector in the standard orthonormal basis) for any $t \in [n]$;
        
        \item since $m$ is even, $[m]^n \times [2n] = A \sqcup B \sqcup C$, hence \cnt;
        
        \item by \eqref{eq:adjacentCells} and since voxels $i$ and $i + e_t$ are adjacent by the coordinate $t$ we have $\sum\limits_{(i,k)\in A} c_{i,k} = \sum\limits_{(i,k)\in B} c_{i,k} = 0$, which entails \cnt.

    \end{itemize}
    In the last element of the equation there are $2nm^{n-1}$ summands $c_{i,k}$.
    Therefore at least one of the summands exceeds
    $\frac{m^n}{2}\cdot\frac{1}{2nm^{n-1}} = \frac{m}{4n}$.
    \aronly{\footnote{Lemma \ref{lem:LongSternfeld array} is formulated for arrays of even size.
    The analogous lemma is true for arrays of odd size as well, with minor changes in the proof;
    however, this is not required for our proof of Theorem~\ref{th:basicSubset}.}}
\end{proof}

\stabilizemcnt

For a continuous function $f\colon K\to\R$ defined on a closed bounded subset $K \subset \R^d$ the norm $\norm{f}$ is the maximum of the absolute value of $f$ over $K$.
    
Denote $|(\fromto{i}{n})| := i_1 + \ldots + i_n$.

\begin{lemma}
\label{lem:boundedFunctionOnArray}
    Let $X \subset \R^{2n}$ be a closed bounded set containing a Sternfeld array $(a_i)_{i \in [4nm]^n}$ of size $4nm$.
    Then there is a continuous function
    \[
        f_m\colon X\to\R\quad\mbox{such that}\quad\norm{f_m} \leqslant 1 \quad\text{and}\quad
        f_m(a_i)=(-1)^{\abs{i}}\quad\mbox{for any}\; i\in[4nm]^n.
    \]
\end{lemma}
\begin{proof}
    Take $(4nm)^n$ disks with centers at $a_i$ and radii
    $\frac{1}{3} \min\limits_{i,j} \lvert a_i - a_j \rvert$,
    where $\abs{x-y}$ is the Euclidean distance in $\R^{2n}$ for $x,y\in\R^{2n}$.
    Outside these disks set $f_m=0$.
    Inside the disk with the center at $a_i$ take $f_m$ to be $(-1)^{\lvert i\rvert}$ at the center, $0$ on the boundary, and extend $f_m$ to the disk linearly in the distance to $a_i$.
\end{proof}

\setcounter{mcnt}{0}
\renewcommand{\themcnt}{\text{(\ref{th:basicSubset}.\alph{mcnt})}}
\begin{proof}[Proof of Sternfeld's Arrays Theorem \ref{th:basicSubset}]
    
    %\emph{First step. Main notations.}
    For each $k \in [2n]$ let $X_k$ be the projection of $X$ to the $k$th axis.
    A continuous function $f \colon X \to \R$
    \emph{has basic decomposition $(\varphi_k\colon X_k\to\R)_{k\in[2n]}$} if $f(x) = \sum\limits_{k=1}^{2n} \varphi_k(x_k)$ for each $x \in X$.
    Note that $\varphi_k$ can be extended to $\R$ without increasing the norm, since $\varphi_k$ is bounded and $X_k$ is closed and bounded, for each $k\in[2n]$.
    Then $X$ is~basic if and only if each continuous function defined on $X$ has some basic decomposition.
    
    \emph{Construction of integers $m_s$ and functions $F_s, \varphi^s_k$ and $F$}.
    For each integer~$m$ take a~Sternfeld array $\bigl(a(m)_i\bigr)_{i \in [4nm]^n}$ of size $4nm$
    and a~function $f_m$ from Lemma~\ref{lem:boundedFunctionOnArray}, applied to~$X$ and $\bigl(a(m)_i\bigr)_{i \in [4nm]^n}$.

    Define integers $m_s$ and continuous functions $F_s\colon X\to\R$ inductively as follows.
    Set $m_0:=1$ and $F_0:=0$.
    Suppose now that $F_s$ and $m_s$ are defined.
    If $F_s$ has no basic decomposition then Sternfeld's Arrays Theorem \ref{th:basicSubset} is proved.
    Otherwise take any basic decomposition $(\varphi^s_k\colon X_k \to \R)_{k \in [2n]}$ of $F_s$.
    %Note that in the definition of a basic function each $\varphi_k$ is defined on all of $\R$.
    %However, since each $X_k$ is closed and bounded, each $\varphi_k^s$ is bounded and can be extended to $\R$ without increasing the norm.
    Take
    \begin{equation}
    \label{eq:inductiveMs}
    \tag{*}
        m_{s+1}>m_s\left(\max\limits_{k \in [2n]}\norm{\varphi^s_k} + s + 1\right)
    \end{equation}
    and
    \[
        F_{s+1}=F_s+\dfrac{f_{m_{s+1}}}{m_s}.
    \]
    
    Construct in this way an infinite number of $m_s$, $F_s$ and $\varphi_k^s$.
    Define the continuous function $F\colon X\to\R$ by
    \[
        F=\lim\limits_{s\to\infty}F_s=\sum\limits_{s=0}^\infty\frac{f_{m_{s+1}}}{m_s}.
    \]
    The function series converges since $\norm{\frac{f_{m_{s+1}}}{m_s}} \leqslant \frac{1}{m_s} < \frac{1}{s!}$ (which follows since $\frac{m_s}{m_{s-1}} > s$ and $m_0=1$).
    
    \emph{Proof that $F$ has no basic decomposition}.
    Assume to the contrary that there is a basic decomposition $(\varphi_k\colon X_k \to \R)_{k \in [2n]}$ of $F$.
    
    It suffices to prove that $\max\limits_{k\in[2n]}\norm{\varphi_k}>s$ for any $s$.
    This inequality follows from
    \[
        \max_k \norm{\varphi_k} + \max_k \norm{\varphi_k^s} \stackrel{\juststep}{\geqslant}
        \max_k \norm{\varphi_k - \varphi_k^s} \stackrel{\juststep}{>}
        \frac{m_{s+1}}{m_s} \stackrel{\juststep}{>}
        \max_k \norm{\varphi_k^s} + s.
    \]
    Here
    \begin{itemize}[nosep]
        \item (\ref{th:basicSubset}.a) is obvious;
        
        \item (\ref{th:basicSubset}.b) follows by
        applying Lemma \ref{lem:LongSternfeld array} to 
        \begin{equation*}
            m=4nm_{s+1}\quad
            % \bigl(a_i\bigr)=\bigl(a(m_{s+1})_i\bigr)
            % \quad
            \text{and}\quad
            c_{i,k} = (-1)^{\abs{i}} m_s (\varphi_k - \varphi_k^s)(a_{i,k}),
            \quad\text{where}\quad
            (a_i)=(a(m_{s+1})_i);
        \end{equation*}
        
        \item (\ref{th:basicSubset}.c) follows by \eqref{eq:inductiveMs}.
        
    \end{itemize}
    It remains to check that conditions \eqref{eq:adjacentCells} and \eqref{eq:halfClose} of Lemma \ref{lem:LongSternfeld array} are fulfilled
    for those $m$ and $(c_i)_{i \in [m]^n}$.
    
    Let us check that \eqref{eq:adjacentCells} is fulfilled.
    For voxels $i,j$ adjacent by the coordinate $t$ we have $(-1)^{\lvert i\rvert} + (-1)^{\lvert j\rvert} = 0$
    and
    $(\varphi_{\xi(t)} - \varphi_{\xi(t)}^s)(a_{i,\xi(t)}) = (\varphi_{\xi(t)} - \varphi_{\xi(t)}^s)(a_{j,\xi(t)})$.
    Thus \eqref{eq:adjacentCells} is fulfilled.
    
    Condition \eqref{eq:halfClose} is fulfilled because
    for
    $s \geqslant 2$ we have
    \begin{multline*}
        \abs{\sum\limits_{k=1}^{2n} c_{i,k} - 1}
        \stackrel{\juststep}{=} \\ \stackrel{\themcnt}{=}
        \abs{(-1)^{|i|}\sum\limits_{k=1}^{2n} c_{i,k} - f_{m_{s+1}}(a_i)}
        \stackrel{\juststep}{=}
        m_s\abs{F(a_i) - F_s(a_i) - \frac{f_{m_{s+1}}(a_i)}{m_s}}
        \stackrel{\juststep}{=} \\ \stackrel{\themcnt}{=}
        m_s\abs{F(a_i) - F_{s+1}(a_i)}
        \stackrel{\juststep}{=}
        m_s\abs{\sum_{l=s+1}^\infty \frac{f_{m_{l+1}}(a_i)}{m_l}}
        \stackrel{\juststep}{\leqslant}
        \sum_{l=s+1}^\infty \frac{m_s}{m_l}
        \stackrel{\juststep}{<} \\ \stackrel{\themcnt}{<}
        \sum_{l=s+1}^\infty \prod_{q=s+1}^{l} \frac{1}{q}
        \stackrel{\juststep}{<}
        \sum_{l=s+1}^\infty \frac{1}{3^{l-s}}
        \stackrel{\juststep}{=}
        \frac{1}{2}.
    \end{multline*}
    Here
    \begin{itemize}[nosep]
        \item (\ref{th:basicSubset}.d) and (\ref{th:basicSubset}.h) follow by definition of $f_m$;
        \item (\ref{th:basicSubset}.e) follows by the choice of $c_{i,k}$ and by definition of a basic decomposition;
        \item (\ref{th:basicSubset}.f), (\ref{th:basicSubset}.g) follow by the inductive definition of $F_s$ and definition of $F$, respectively;
        \item (\ref{th:basicSubset}.i) follows since $\frac{m_{l-1}}{m_l} < \frac{1}{l}$, which follows by \eqref{eq:inductiveMs};
        \item (\ref{th:basicSubset}.j) follows since $s + 1 \geqslant 3$;
        \item (\ref{th:basicSubset}.k) is obvious.
    \end{itemize}

\end{proof}

\stabilizemcnt

\aronly{
\section{Open problems and discussion}\label{s:problems}

\begin{remark}[Relation to other proofs]
\label{rem:otherProofs}
\hspace{0pt}
\begin{remarkenumi}
    \item \label{cnt:comparisonWithSternfeld}    
    In his implicit proof of Theorem~\ref{th:basicSubset} Sternfeld used the bounded inverse theorem
    (\cite[\S1, (1.10)]{St85}, \cite[\S2, Theorem~10 and the argument above]{St89})
    and some combinatorial argument (\cite[\S2, Proposition~2.1]{St85}, \cite[\S5, item~12]{St89}).
    Our simpler proof of Theorem~\ref{th:basicSubset} follows the idea of \cite{MKT} and is based on the Weierstrass M-test for convergence of function series and Lemma \ref{lem:LongSternfeld array}, respectively. 
Thus the proofs in \cite{MKT} and in this paper use the direct language of functions instead of the dual language of measures (used by Sternfeld). 
Although these languages are dual (and so analogous), they are not quite parallel. 
In particular, the direct language of functions gives an explicit construction of a non-basic function.  
%having in their basic representation functions of infinite norm 
%is a simplification which does not lose ??? in our particular case.
    Cf.~\cite[\S2, the second paragraph after Theorem~2]{MKT}.
    
    Sternfeld's implicit proof of Theorem~\ref{th:basicSubset} takes up about two and a half pages in \cite[proofs of Proposition~2.1 and Theorem~5]{St85}, or with details four pages in \cite[pp.~36--39]{St89}.
    Our proof is as detailed as \cite{St89} and takes up about two and a half pages.
    
    \item \label{cnt:comparisonWithMKT}
    A proof of Theorem~\ref{th:basicSubset} for $n=1$, which is simpler than Sternfeld's, is given in~\cite{MKT}.
    A simpler exposition of the proof from~\cite{MKT} is given in \cite[p.~9]{Sk}.
    Our proof of Theorem~\ref{th:basicSubset} is a high-dimensional generalization of the proof from~\cite[p.~9]{Sk}.
    More precisely, Lemma~\ref{lem:LongSternfeld array} is a high-dimensional generalization of combinatorial lemmas \cite[\S3, Lemma~2]{MKT}, \cite[p.~9, Lemma]{Sk}.
    The deduction of Theorem~\ref{th:basicSubset} from Lemma~\ref{lem:LongSternfeld array} is a simple generalization of the deductions of Theorem \ref{th:basicSubset} in the case $n=1$ from the combinatorial lemmas in \cite[\S3, proof of Theorem~3]{MKT}, \cite[p.~9, proof of the `only if' part]{Sk}.
    
    \item 
    Our proof slightly simplifies the proof from \cite[p.~9]{Sk}.
    Precisely, in condition \eqref{eq:halfClose} of Lemma \ref{lem:LongSternfeld array} we need
    the values to be $\frac{1}{2}$-far from~$0$;
    in the combinatorial lemma from \cite[p.~9]{Sk} one needs $\frac{2}{m-5}$-closeness to~$\pm 1$, where $m$ is the size of the Sternfeld array.
    Also, in inequality \eqref{eq:inductiveMs} we have $m_s$, while in the analogous inequality from \cite[p.~9]{Sk} one has $m_s!$.
    
    \item \label{cnt:compareDefs}
    In this item we show that the definition of a Sternfeld array is a specific case of \cite[Definition~2.4]{St85} for arrays of large enough size.
        
    Below notions $c$, $n$, $X$, $Y_i$, $T^*$, $\varphi_i$, $\mu$, $\delta$, $L_i$, $E_i$ and $\sigma$ are from \cite[Definition~2.4]{St85}.
    Notions $(a_i)_{i \in [m]^n}$ and $x_k$ are from our paper.
    By $\abs{Z}$ is denoted the cardinality of a set $Z$.
    
    In order to get a specific case of \cite[Definition~2.4]{St85} from our paper, one should
    replace in \cite[Definition~2.4]{St85} the size $m$ of an array by $m^n$, $[m]$ by $[m]^n$, $\{x_j\}_{j=1}^m$ by $(a_i)_{i\in[m]^n}$, $i$ by $k$,
    and should set in \cite[Definition~2.4]{St85} $c = 2n$, $X = \R^{2n}$, $T^* = [2n]$, $Y_k = \R$ and $\varphi_k(x) = x_k$ for $k \in [2n]$. Then
        
    \cite[(ar.1)]{St85} holds since $\mu = \sum_{i \in [m]^n} (-1)^{i_1 + \ldots + i_n} \delta_{a_i}$;
        
    \cite[(ar.2')]{St85} holds since points of a Sternfeld array are pairwise distinct;
        
    \cite[(ar.3)]{St85} holds since
        for odd $m$
        \begin{align*}
            L_{2t-1} &= [m] ^ {t - 1} \times \{1, \ldots, m-1\} \times [m]^{n - t}, \\
            L_{2t} &= [m] ^ {t - 1} \times \{2, \ldots, m\} \times [m]^{n - t},
        \end{align*}
        and for even $m$
        \begin{align*}
            L_{2t-1} &= [m]^n, \\
            L_{2t} &= [m] ^ {t - 1} \times \{2, \ldots, m-1\} \times [m]^{n - t};
        \end{align*}
        
    \cite[(ar.3.1')]{St85} holds since the decompositions $E_k$ of $L_k$ may be easily gotten from Figure~\ref{fig:arrayGridR4} and the form of the sets $A$, $B$, $C$ in equality~(\ref{lem:LongSternfeld array}.b);
        
    \cite[(ar.3.2)]{St85} holds since
    \begin{align*}
        %\text{for odd $m$}\quad&
        &\sigma(j) = \bigl\{2t-1 : t \in [n],\; j_t \neq m \bigr\}
        \cup \bigl\{2t : t \in [n],\; j_t \neq 1 \bigr\}
        \quad&\text{for $j \in [m]^n$ and $m$ odd,} \\
        %\text{for even $m$}\quad&
        &\sigma(j) = \bigl\{2t-1 : t \in [n]\bigr\}
        \cup \bigl\{2t : t \in [n],\; j_t \neq 1,m \bigr\}
        \quad&\text{for $j \in [m]^n$ and $m$ even,} \\
        %\text{and}\quad&
        &\abs{\bigl\{ j \in [m]^n : \abs{\sigma(j)} = 2n \bigr\}} = (m-2)^n \geqslant m^n - c m^{n-1}
        \quad&\text{as $m \geqslant 3^n$.}
    \end{align*}
    
    \item 
    The definition of a Sternfeld array in the plane (Remark~\ref{rem:definitionPlane} or the definition of a Sternfeld array in $\R^{2n}$ in the case $n=1$) is a particular case of \cite[\S2, Definition~2]{MKT}.
    The difference is that
    \begin{itemize}
        \item in our paper a Sternfeld array is a finite sequence, while in \cite[\S2, Definition~2]{MKT} an array can be a countable sequence;
        \item in our paper the segment $[\,a_1,a_2\,]$ is parallel to the ordinate axis, while in \cite[\S2, Definition~2]{MKT} this segment is parallel just to one of the abscissa or ordinate axes;
        \item in our paper all points $a_i$ are pairwise distinct, while in \cite[\S2, Definition~2]{MKT} only each two consecutive points $a_i$ and $a_{i+1}$ are distinct.
    \end{itemize}

\end{remarkenumi}    
\end{remark}

\begin{remark}[Generalizations]
\label{rem:generalSternfeld}
    Lebesgue covering dimension \cite{HW} $\dim X$ of a topological space $X$ is the minimal~$n$ such that for any open cover of $X$ there is an open refinement of this cover such that each point of $X$ belongs to at most $n+1$ sets of the refinement.
    
    Ostrand proved the result, analogous to Kolmogorov's Theorem (Remark~\ref{rem:motivation}.\ref{en:motivation:th-kolm-super}), for any compact metric space of Lebesgue dimension $n > 1$ (see the survey \cite{St89}).
    
    Sternfeld's Dimension Theorem (Remark~\ref{rem:motivation}.\ref{en:motivation:th-stern-dim}) has the analogous generalization \cite{St85, St89}.
    
\end{remark}

\begin{remark}[Discussion of the solution of Arnold's Problem~\ref{prob:Arnold}]
\label{rem:ArnoldDiscussion}
    
    Denote by $\abs{Z}$ the~number of elements in a~finite set~$Z$.
    For each $k \in [n]$ define the~projection $\pi_k\colon\R^n\to\R$ by $\pi_k(x) := x_k$.
    For a~subset~$Y \subset \R^n$
    denote
    \[
        E(Y) := \left\{
        x \in Y :
        \abs{\pi_k^{-1}(x_k) \cap Y} \geqslant 2 \;\text{for all}\; k \in [n]
        \right\}.
    \]

    For a closed bounded subset $X \subset \R^n$ consider the following properties (cf. \cite[\S1, Theorem~2, (B), (E), (A)]{MKT} for the plane):
    \begin{align*}
        (\text{bas}_n) \qquad&
        X \quad\text{is basic}. \\
        (\text{E}_n) \qquad&
        E^l(X) = E\Bigl(E\bigl(\ldots E(X)\ldots\bigr)\Bigr) = \varnothing \quad\text{for some}\quad l\in\N. \\
        (\text{arr}_{2n}) \qquad&
        %d\quad\text{is even and} \\
        X \subset \R^{2n} \quad \text{does not contain Sternfeld arrays of arbitrary large size}.
    \end{align*}

    In \cite[\S 2, Lemma~23, \emph{(ii)}]{St89}
    Sternfeld proved that $(\text{E}_n) \implies (\text{bas}_n)$.
    Theorem~\ref{th:basicSubset} means that $(\text{bas}_{2n}) \implies (\text{arr}_{2n})$.

    Sternfeld solved Arnold's Problem~\ref{prob:Arnold} by proving that $(\text{E}_2) \Longleftrightarrow (\text{bas}_2)$
    \cite[\S 2, Lemma~23]{St89}.
    Let us denote by $(\text{arr}_2')$ the property $(\text{arr}_2)$ with a following change in the definition of a Sternfeld array in the plane (Remark~\ref{rem:definitionPlane}): `pairwise distinct points in $\R^2$ such that' should be replaced by `points in $\R^2$ such that for all $i\in\{1, \ldots, m-1\}$ the points $a_i$ and $a_{i+1}$ are distinct,'.
    The equivalence $(\text{E}_2) \Longleftrightarrow (\text{arr}_2')$ is obvious (see e.g. \cite[p.~8, item~10]{Sk}).

    The inverse implication $(\text{bas}_n) \implies (\text{E}_n)$ is false \cite[Example~4.8]{NR}.
    The inverse implication $(\text{arr}_{2n}) \implies (\text{bas}_{2n})$ is probably also false.

    The implication $(\text{E}_{2n}) \implies (\text{arr}_{2n})$ is a simple high-dimensional generalization of \cite[p.~8, item 10]{Sk}.
    For the proof one may need the following change in the definition of a~Sternfeld array:
    %in the end we should add `up to the swap of axes $2t-1$ and $2t$ simultaneously for all $i_1, \ldots, i_{t-1}, i_{t+1}, \ldots, i_n$'.
    at the~beginning `An $n$-dimensional array $(a_i)_{i\in [m]^n}$' should be replaced by `Let $I \subset \Z$ be a finite set of indicies, of cardinality $m$. An $n$-dimensional array $(a_i)_{i\in I^n}$'.
    
    The definition of a Sternfeld array can be easily given for odd dimensions as well;
    however, this is not required for Sternfeld's Dimension Theorem (Remark~\ref{rem:motivation}.\ref{en:motivation:th-stern-dim}).

\end{remark}

\begin{problem}
    It is interesting to find a direct proof of the statement
    \emph{if~a~closed bounded subset $X \subset \R^{2n}$ contains arrays (as defined in \cite[Definition~2.4]{St85}) of arbitrary large size then $X$ is not basic.}
\end{problem}

}

\end{document}

%% file: eng.tex
\usepackage[english]{babel}
\usepackage[T1]{fontenc}
\usepackage[utf8]{inputenc}

%% file: eng_theorems.tex
\theoremstyle{plain}
    \newtheorem{theorem}{Theorem}[section]
    \newtheorem{lemma}[theorem]{Lemma}

\theoremstyle{definition}
    \newtheorem{remark}[theorem]{Remark}
    \newtheorem{problem}[theorem]{Problem}

\usepackage[
    draft = false,
    unicode = true,
    colorlinks = true,
    allcolors = blue,
    hyperfootnotes = true,
    citecolor = red
]{hyperref}

%% file: preamble.tex
\usepackage{amsfonts,amsmath,amssymb,amscd,amsthm,url}
\usepackage[inline]{enumitem}
\usepackage{graphicx,epsf,afterpage}
\usepackage[all]{xy}
\usepackage{multicol}
\usepackage{array}
\usepackage{braket}
\usepackage{epigraph}

\usepackage{ulem}

\usepackage{tikz}
\usetikzlibrary{positioning, calc, intersections, through, arrows}
\usetikzlibrary{graphs}
\usetikzlibrary{graphs.standard}
\usetikzlibrary{patterns}
\long\def\comment#1\endcomment{}

\newcommand\norm[1]{\ensuremath{\left\lVert#1\right\rVert}}
\newcommand\abs[1]{\ensuremath{\left\lvert#1\right\rvert}}

\newcommand{\R}{\ensuremath{\mathbb{R}}}
\newcommand{\Z}{\ensuremath{\mathbb{Z}}}

\newcommand{\N}{\ensuremath{\mathbb{N}}}

\newcounter{mcnt}
\renewcommand{\themcnt}{(\alph{mcnt})}
\newcommand{\cnt}{\addtocounter{mcnt}{1}\themcnt\ \nopagebreak}
\newcommand{\juststep}{\addtocounter{mcnt}{1}\themcnt}
\newcommand{\stabilizemcnt}{\renewcommand{\themcnt}{(\alph{mcnt})}}

\newcounter{wordcnt}

%% file: listmain.tex
\newcommand{\setenumi}{
    \topsep=0mm
    \partopsep=0mm
    \parsep=0mm
    \itemsep=0mm
    \leftmargin=0pt
    \listparindent=7mm
    \itemindent=3mm
    \labelsep=2mm
    \labelwidth=-5mm
    \usecounter{enumi}
    }

\newenvironment{remarkenumi}
    {\begin{list}{\labelenumi}{\setenumi}}
    {\end{list}}

\renewcommand{\labelenumi}{(\alph{enumi})}

%% file: main.bbl
\begin{thebibliography}{MMM+}

\bibitem[ADN+]{ADN+}
\emph{E.~Alkin, S.~Dzhenzher, O.~Nikitenko, A.~Skopenkov, A.~Voropaev},
Cycles in graphs and hypergraphs.
The 35th Summer conference of the International mathematical Tournament of towns.
\url{https://arxiv.org/abs/2308.05175}



\bibitem[Ar]{Ar}
\emph{V.~I. Arnold},
Problem 6 (in Russian). Mat. Prosveshenie, ser. 2, 1958, 273--274.
%\url{http://ilib.mccme.ru/djvu/mp2/mp2-3.djvu?djvuopts&page=274}


\bibitem[CRS]{CRS}
\emph{A.~Cavicchioli, D.~Repov\v s and A.~B.~Skopenkov},
Open problems on graphs, arising from geometric topology. Topol. Appl. 84 (1998) 207--226.

\jonly{
\bibitem[Dz22]{Dz22}
\emph{S.~Dzhenzher},
A simpler proof of Sternfeld's Theorem.
\url{https://arxiv.org/abs/2110.15565}
}

\bibitem[HW]{HW}
\emph{W.~Hurewicz and H.~Wallman},
Dimension Theory. Princeton, 1941.


\bibitem[Ku]{Ku}
\emph{V.~Kurlin},
Basic embeddings into a product of graphs.
Topology and its Applications 102 (2000) 113--137.
%\url{https://doi.org/10.1016/S0166-8641(98)00147-3}

\bibitem[Ko]{Ko}
\emph{A.~N.~Kolmogorov},
On representation of continuous functions of several variables by superpositions of functions of less variables and addition (in Russian).
Dokl. AN USSR, 114:5, 953--965 (1957).
%\url{http://www.mathnet.ru/links/90e35c02a59c4c5009b53e9ca17ab3a2/dan22050.pdf}


\bibitem[Mi]{Mi}
\emph{E.~Miliczk\'a},
Constructive decomposition of a function of two variables as a sum of functions of one variable.
Proc. Amer. Math. Soc. 137 (2009), 607--614.

\bibitem[MKT]{MKT}
\emph{N.~Mramor-Kosta, E.~Trenklerov\'a},
On basic embeddings of compacta into the plane.
Bull. Austral. Math. Soc. 68 (2003) 471--480.

\bibitem[NR]{NR}
\emph{K.~Nurligareev, I.~Reshetnikov},
Decompositions of functions defined on finite sets in $\R^d$.
Journal of Knot Theory and Its Ramifications. 31:2 (2022)
\url{https://arxiv.org/abs/2204.11084}
 

\bibitem[Re]{Re}
\emph{I.~Reshetnikov},
Decomposition of number arrangements in the cube, in Russian.
\newline
\url{https://arxiv.org/abs/1412.8078}

\bibitem[Sk]{Sk}
%\emph{А. Скопенков}, Базисные вложения и 13-я проблема Гильберта,
%Мат. Просвещение, 14 (2010), 143--174.
\emph{A.~Skopenkov},
Basic embeddings and Hilbert's 13th problem (in Russian), Mat. Prosveshenie, 2010
\url{http://arxiv.org/abs/1001.4011}
Abridged English translation:
%\linebreak
\url{http://arxiv.org/abs/1003.1586}


\bibitem[St85]{St85}
\emph{Y.~Sternfeld}, Dimension, superposition of functions and separation of points in compact metric spaces. Israel J. Math. 50 (1985) 13--53.

\bibitem[St89]{St89}
\emph{Y.~Sternfeld}, Hilbert's 13th problem and dimension, Lect. Notes Math. 1376 (1989) 1--49.


\end{thebibliography}
